\documentclass{amsart}

\usepackage{latexsym}
\usepackage{amssymb}
\usepackage{amsmath,mathabx}

\newtheorem{theorem}{Theorem}[section]
\newtheorem{lemma}[theorem]{Lemma}
\newtheorem{proposition}[theorem]{Proposition}

\theoremstyle{definition}
\newtheorem{definition}[theorem]{Definition}

\theoremstyle{remark}
\newtheorem{remark}[theorem]{Remark}

\numberwithin{equation}{section}

\newcommand{\R}{\mathbb{R}}
\newcommand{\Z}{\mathbb{Z}}
\newcommand{\N}{\mathbb{N}}

\newcommand{\f}{\varphi}

\newcommand{\ep}{\varepsilon}

\newcommand{\B}{\mathcal B}
\newcommand{\F}{\mathcal F}
\newcommand{\Pp}{\mathbb P}
\newcommand{\Rr}{\mathcal R}
\newcommand{\Om}{\Omega}

\newcommand{\Ss}{\mathcal S}

\begin{document}

\title[Haagerup property and infinite measures]{A new characterization of the Haagerup property by  actions on infinite measure spaces}

\author{Thiebout Delabie\ $^\dag$, Paul Jolissaint\ $^\ddag$, Alexandre Zumbrunnen\ $^\ddag$}
\address{Universit\'e Paris-Sud,
Facult\'e des Sciences d'Orsay,
D\'epartement de Math\'ematiques,
B\^atiment 307,
F-91405 Orsay Cedex}
\email{thiebout.delabie@gmail.com}
\address{Universit\'e de Neuch\^atel,
       Institut de Math\'ematiques,       
       E.-Argand 11,
       2000 Neuch\^atel, Switzerland}
       
\email{paul.jolissaint@unine.ch,}
\email{alexandre.zumbrunnen@unine.ch}


\date{\today}

\keywords{Locally compact groups, unitary representations, Haagerup property\\
$^\dag$Universit\'e Paris-Sud, T. D. is supported by grant P2NEP2 181564 of the Swiss National Science Foudation, \ $^\ddag$Universit\'e de Neuch\^atel}

\begin{abstract}
The aim of the article is to provide a characterization of the Haagerup property for locally compact, second countable groups in terms of actions on $\sigma$-finite measure spaces. It is inspired by the very first definition of amenability, namely the existence of an invariant mean on the algebra of essentially bounded, measurable functions on the group.
\end{abstract}

\maketitle

\section{Introduction}

Throughout this article, $G$ denotes a locally compact, second countable group (lcsc group for short); we assume furthermore that it is non-compact, because the compact case is not relevant to the Haagerup property \cite{ccjjv} which is the central theme of these notes. 

The latter property is often interpreted as a weak form of amenability or as
a strong negation of Kazhdan's property (T), depending on the context.

A way to see that it is a weak form of amenability is to consider the following characterization: recall that $G$ has the Haagerup property if and only if there exists a sequence of normalized, positive definite functions $(\f_n)_{n\geq 1}$ on $G$ such that $\f_n$ converges to the constant function $1$ uniformly on compact subsets of $G$ as $n\to\infty$, and each $\f_n\in C_0(G)$, \textit{i.e.} $\f_n\to 0$ at infinity. In turn, $G$ is amenable if and only if each $\f_n$ can be chosen with compact support.
\par\vspace{2mm}

Nowadays, there are several characterizations of the Haagerup property: apart from the ones presented in the monography \cite{ccjjv}, we can mention for instance the ones involving actions on median spaces or on measured walls as in  \cite{chatterjietall}. 
In fact, the characterization presented here has no direct relationship with the latter ones; it rests rather on strongly mixing actions on probability spaces: see Theorem \ref{thm2.2.2}.

\par\vspace{2mm}
The present article has its origin in the characterization of amenability given by its original definition: the lcsc group $G$ is amenable if and only if the algebra $L^\infty (G)$ has a $G$-invariant mean with respect to the action of $G$ on itself by left translation. 

Observe that such an action is a special case of proper actions that we recall now.
Let $\Om$ be a locally compact space. Suppose that the lcsc group $G$ acts continuously on $\Om$. Then the action is \textit{proper} if for all compact subsets $K,L\subset\Om$, the set
\[
\{g\in G : gK\cap L\not=\emptyset\}
\]
is relatively compact in $G$. Then a question arises: Which lcsc groups admit proper actions on locally compact spaces $\Om$ equipped with an invariant measure and such that $L^\infty(\Om)$ has an invariant mean? Here is the answer:

\begin{proposition}
Let $G$ be a lcsc group and let $\Omega$ be a locally compact space on which $G$ acts properly and which admits a $G$-invariant, regular Borel measure $\mu$. If $L^\infty(\Omega)$ has an invariant mean, then $G$ is amenable.
\end{proposition}

Indeed, it is a well-known fact that the existence of an invariant mean on $L^\infty(\Omega)$ is equivalent to the fact that, for every compact set $L\subset G$ and every $\ep>0$, there exists a continuous function $\xi$ on $\Om$ with compact support, such that $\Vert\xi\Vert_2=1$ and
\[
\sup_{g\in L}|\langle\pi_\Om(g)\xi|\xi\rangle-1|<\ep
\]
where $\pi_\Om$ denotes the natural unitary representation of $G$ on $L^2(\Om)$ associated to the action of $G$ on $\Om$.
As the latter is proper and $\xi$ has compact support, the coefficient function $\f=\langle\pi_\Om(\cdot)\xi|\xi\rangle$ has compact support too, and thus the constant function $1$ is a uniform limit on compact sets of compactly supported positive definite functions, which means that $G$ is amenable.

\par\vspace{2mm}
Thus, $G$ is amenable if and only if it admits a proper, measure-preserving action on some locally compact space $\Om$ so that $L^\infty(\Om)$ admits a $G$-invariant mean.

\par\vspace{2mm}
As properness of actions is too strong to characterize the Haagerup property, we consider the setting of measure-preserving actions on measure spaces equipped with invariant measures. It turns out that the following property is well adapted to our situation.

\begin{definition}\label{C0dynam}
Let $(\Om,\B,\mu)$ be a measure space on which a lcsc group $G$ acts by Borel automorphisms which preserve $\mu$. Then we say that the corresponding dynamical system $(\Om,\B,\mu,G)$ is a $C_0$-\textit{dynamical system} if, for all $A,B\in \B$ such that $0\le\mu(A),\mu(B)<\infty$, one has
\[
\lim_{g\to\infty}\mu(gA\cap B)=0.
\]
\end{definition}

\begin{remark}\label{nonergodic}
Let $(\Omega,\B,\mu,G)$ be a $C_0$-dynamical system. Then every $G$-invariant set $A\in\B$ such that $\mu(A)<\infty$ is automatically of measure zero. In particular, if $\mu$ is finite then it is equal to zero.
Moreover, the action of $G$ is not ergodic in general. Indeed, let $\Z$ act by translations on $\R$ equipped with Lebesgue measure $\mu$. Then 
\[
A:=\{x+k : x\in (0,1/2),\ k\in\Z\}=\bigsqcup_{k\in\Z} (k,k+1/2)
\]
is $\Z$-invariant and $\mu(A)=\mu(A^c)=\infty$. The action is $C_0$ since it is proper, as it is the restriction to $\Z$ of the action of $\R$ on itself.
\end{remark}

\begin{remark}
Let $(\Om,\B,\mu,G)$ be a (measure-preserving) dynamical system; then it is a $C_0$-dynamical system if and only if the permutation representation $\pi_\Om$ of $G$ on $L^2(\Om)$ is a $C_0$-representation. Hence, we infer that if $G$ admits a $C_0$-dynamical system $(\Om,\B,\mu,G)$ such that $L^\infty(\Om)$ has a $G$-invariant mean, then $G$ has the Haagerup property.
\end{remark}

Then the goal of the present article is to prove the converse.

\begin{theorem}\label{thmC0}
Let $G$ be a lcsc group which has the Haagerup property. Then there exists a $C_0$-dynamical system $(\Om,\B,\mu,G)$ such that $L^\infty(\Om)$ has a $G$-invariant mean. More precisely, 
\begin{enumerate}
\item [(1)] the measure $\mu$ is $\sigma$-finite, $G$-invariant and the Hilbert space $L^2(\Om,\mu)$ is separable;
\item [(2)] for all measurable sets $A,B\in \B$ such that $0\leq\mu(A),\mu(B)<\infty$, we have
\[
\lim_{g\to\infty}\mu(gA\cap B)=0;
\]
in other words, $\pi_\Om$ is a $C_0$-representation;
\item [(3)] there exists a sequence of unit vectors $(\xi_n)\subset L^2(\Om,\mu)$ such that $\xi_n\ge 0$ for every $n$, and for every compact set $K\subset G$, one has
\[
\lim_{n\to\infty}\sup_{g\in K}\,\langle\pi_\Om(g)\xi_n|\xi_n\rangle=1;
\]
in other words, the dynamical system $(\Om,\B,\mu,G)$ has an invariant mean.
\end{enumerate}
\end{theorem}

\par\vspace{2mm}
The proof of Theorem \ref{thmC0} will occupy the rest of the article. It relies on the characterization of the Haagerup property stated in Theorem 2.2.2 of \cite{ccjjv} that we recall now.

\begin{theorem}\label{thm2.2.2} (\cite{ccjjv}, Theorem 2.2.2)
Let $G$ be a lcsc group. Then it has the Haagerup property if and only if there exists a standard probability space $(S,\B_S,\nu)$ on which $G$ acts by Borel automorphisms which preserve $\nu$, and $(S,\B_S,\nu)$ has the following additional two properties:
\begin{enumerate}
\item [(a)] the action of $G$ on $S$ is \textit{strongly mixing}, which means that for all $A,B\in \B_S$,
\[
\lim_{g\to \infty}\nu(gA\cap B)=\nu(A)\nu(B);
\]
\item [(b)] the action admits a non-trivial asymptotically invariant sequence: there exists a sequence $(A_n)_{n\geq 1}\subset \B_S$ such that $\nu(A_n)=1/2$ for every $n$ and such that, for every compact set $K\subset G$,
\[
\lim_{n\to\infty}\sup_{g\in K}\nu(gA_n\bigtriangleup A_n)=0.
\]
\end{enumerate}
\end{theorem}

Furthermore, we assume that $S$ is a compact metric space on which $G$ acts continuously, and $\nu$ has support $S$, according to (the proof of) Lemma 1.3 of \cite{AEG} that we recall for the reader's convenience.

\begin{lemma}\label{Giordanoetall}
(\cite{AEG}, Lemma 1.3)
Let $X$ be a standard Borel $G$-space with a $G$-invariant probability measure $\mu$. Then there exists a compact metric space $Y$, on which $G$ acts continuously, and a $G$-invariant probability measure $\nu$, whose support is $Y$, such that $L^2(X,\mu)$ and $L^2(Y,\nu)$ are $G$-isomorphic.
\end{lemma}

\par\vspace{2mm}
We end this introduction with a brief sketch of the proof of Theorem \ref{thmC0}. It uses elementary measure theory (see \cite{Bil} on this subject) but it is quite long and involved. 

We start with the infinite product space $X=\prod_{n\geq 1} S$, where $S$ satisfies all properties of Theorem \ref{thm2.2.2} and Lemma \ref{Giordanoetall}. We equip $X$ with the diagonal action of $G$ and with a suitable family $\F$ of subsets containing all sets of the form $B=\prod_n B_n$ such that the infinite product $\prod_n 2\nu(B_n)$ converges, where $B_n\in\B_S$ for every $n$. We construct a measure $\mu$ on $\sigma(\F)$, the $\sigma$-algebra generated by $\F$, that satisfies $\mu(B)=\prod_n 2\nu(B_n)$ for every $B$ as above.

One of the reasons of the choices of such sets and measure $\mu$ is that one can extract from the family $(A_n)$ in condition (b) of Theorem \ref{thm2.2.2} a sequence $(X_m)_m\subset \F$ so that the associated sequence of unit vectors $\xi_m=\chi_{X_m}$ satisfies condition (3) of Theorem \ref{thmC0}. They are unit vectors since $\nu(A_n)=1/2$ for every $n$.

In order to prove that condition (2) holds, consider two subsets $A=\prod_nA_n$ and $B=\prod_n B_n$ such that the infinite products $\prod_n2\nu(A_n)$ and $\prod_n 2\nu(B_n)$ converge. Given $\ep>0$, choose first $N$ large enough so that 
\[
\frac{1}{2}-\ep<\nu(A_n),\nu(B_n)<\frac{1}{2}+\ep
\]
for every $n\ge N$. Then, 
using the strong mixing property of the action of $G$ on $S$ (condition (a) in Theorem \ref{thm2.2.2}), we choose a suitable positive number $\ep'>0$, an integer $m>0$, and a compact set $K\subset G$ such that, for every $g\notin K$,  $\nu(gA_n\cap B_n)\le \nu(A_n)\nu(B_n)+\ep'$ for all $N\le n\le N+m$. Hence we get  
\begin{align*}
\mu(gA\cap B)
&\leq \prod_{n=1}^{N-1}2\nu(A_n)\cdot\prod_{n>N+m}2\nu(A_n)
\cdot
\prod_{n=N}^{N+m} 2(\nu(A_n)\nu(B_n)+\ep')\\
&=
\mu(A)\prod_{n=N}^{N+m}\frac{2\nu(A_n)\nu(B_n)+2\ep'}{2\nu(A_n)}<\ep
\end{align*} 
since the quotients $\frac{2\nu(A_n)\nu(B_n)+2\ep'}{2\nu(A_n)}$ belong to some interval $(0,\delta)$ for a convenient value of $\delta<1$ such that $\delta^{m+1}<\ep/\mu(A)$.
\par\vspace{2mm}

It turns out that the $\sigma$-algebra $\sigma(\F)$ generated by $\F$ is too large, so that the dynamical system $(X,\sigma(\F),\mu,G)$ is not $\sigma$-finite, and, as observed by A. Calderi and A. Valette (cf. Remark \ref{nonContinue}), the associated representation $\pi_X$ on $L^2(X,\mu)$ is not continuous.

Thus, we need to divide the proof of Theorem \ref{thmC0} into two parts: in the first one, we prove that $(X,\sigma(\F),\mu,G)$ is a $C_0$-dynamical system as stated in Definition \ref{C0dynam} on the one hand, and we construct a sequence of unit vectors that satisfy condition (3) in Theorem \ref{thmC0} on the other hand. In the case where $G$ is discrete, it is very simple to restrict our dynamical system to a $\sigma$-finite one, so that the proof is complete with the former additional assumption.
All this is contained in Section 2. 

In the last part of the proof, which is the subject of Section 3, we define a sub-$\sigma$-algebra $\sigma(\F_c)$ of $\sigma(\F)$ and a measure $\mu_c$ so that, for every $A\in\sigma(\F_c)$, $\mu_c(A)<\infty$, we have $\lim_{g\to e}\mu_c(gA\bigtriangleup A)=0$. This implies the continuity of the permutation representation $\pi_X:G\rightarrow U(L^2(X,\sigma(\F_c),\mu_c))$, and finally the latter property is used to prove that we can restrict our dynamical system to get a $\sigma$-finite measure.

\par\vspace{2mm}
\textit{Acknowledgements.} We very grateful to A. Calderi and A. Valette for the content of Remark \ref{nonContinue}, to S. Baaj for his help on the automatic continuity of the permutation representations on separable $L^2$-spaces, and to the referee for his/her valuable comments.

\section{Proof of Theorem \ref{thmC0}, Part 1}

For the rest of the article, $G$ denotes a (non compact) lcsc group with the Haagerup property. According to Theorem \ref{thm2.2.2} and Lemma \ref{Giordanoetall}, let $(S,\B_S,\nu)$ be a compact metric space equipped with a probability measure $\nu$ whose support is $S$, and $G$ acts continously on $S$ and preserves $\nu$ and which satisfies conditions (a) and (b) of Theorem \ref{thm2.2.2}. 

Then put $X=\prod_{n\geq 1}S=\{(s_n)_{n\geq 1}: s_n\in S\ \forall\ n\}$. If $\mathcal{S}$ is any non-empty family of subsets of $X$, we denote by $\sigma(\mathcal S)$ the $\sigma$-algebra generated by $\mathcal S$.

\par\vspace{2mm}
Here is the starting point of our construction.

\begin{definition}\label{famillesF}
Let $X$ be as above. 
\begin{enumerate}
\item [(1)] We denote by $\F_0$ the family of subsets of $X$ 
 of the form $A=\prod_{n\geq 1}A_n$ where $A_n\in \B_S$ for all $n$ such that the infinite product $\prod_{n=1}^\infty 2\nu(A_n)$ exists, \textit{i.e.} the sequence of partial products $(\prod_{n=1}^N 2\nu(A_n))_{N\geq 1}$
converges to some limit in $[0,\infty)$\footnote{Recall that an infinite product $\prod_n u_n$ of complex numbers \textit{converges} if $\lim_{N\to\infty}\prod_{n=1}^N u_n$ exists \textit{and} is different from $0$, and it \textit{converges trivially} if the limit equals $0$. Thus we require that the infinite product $\prod_{n=1}^\infty 2\nu(A_n)$ either converges or converges trivially.}. We also set
\[
\F_{0,+}=\{B=\prod_nB_n\in\F_0 : \nu(B_n)>0\ \forall\ n\}.
\]
\item [(2)] We define the following sequence $(\F_n)_{n\geq 1}$ of collections of subsets  of $X$ by induction: for $n\geq 1$, set $\F_n=\{B\setminus A: A,B\in\F_{n-1}\}$. Finally, we set 
\[
\F:=\bigcup_{n\geq 0}\F_n.
\]
\end{enumerate}
\end{definition}

Observe that $\emptyset\in\F_0$, hence also that $\F_n\subset \F_{n+1}$ for every $n$.

\begin{lemma}\label{2.2}
For all $A,B\in\F_0$, one has $A\cap B\in\F_0$.
\end{lemma}
\begin{proof} Let $A,B\in\F_0$. Let us write $A=\prod_nA_n$ and $B=\prod_nB_n$ as above. Then $\nu(A_n\cap B_n)\leq \min(\nu(A_n),\nu(B_n))$ for every $n$. If there exists an integer $n$ such that $\min(\nu(A_n),\nu(B_n))=0$, then the product $\prod_{n=1}^\infty 2\nu(A_n\cap B_n)$
 converges trivially, and $A\cap B\in\F_0$. Suppose then that $\nu(A_n)>0$ for every $n$, and consider the product of conditional probabilities
\[
a_N=\prod_{n=1}^N\nu(B_n|A_n)=\prod_{n=1}^N\frac{\nu(A_n\cap B_n)}{\nu(A_n)}.
\]
As $0\leq \nu(B_n|A_n)\leq 1$ for every $n$, one has $0\leq a_{N+1}\leq a_N\leq 1$ for every $N$, and the bounded, decreasing sequence $(a_N)_{N\geq 1}$ converges, say, to $a\in [0,1]$. Then the sequence
\[
\prod_{n=1}^N2\nu(A_n\cap B_n)=a_N\cdot\prod_{n=1}^N2\nu(A_n)
\]
converges to $a\cdot\prod_{n=1}^\infty 2\nu(A_n)$.
\end{proof}

\begin{lemma}\label{2.3}
The family $\F$ is a semiring of subsets of $X$, i.e.
\begin{enumerate}
\item [(i)] if $A,B\in\F$ then $A\cap B\in \F$;
\item [(ii)] if $A,B\in\F$ then $B\setminus A\in\F$.
\end{enumerate}
In particular, the set of all finite, disjoint unions of elements of $\F$ is a ring of subsets of $X$. It is the ring generated by $\F$ and is denoted by $\Rr(\F)$.
Moreover,
\begin{enumerate}
\item [(iii)] for every $A\in\F$, there exists $B\in\F_0$ such that $A\subset B$.
\end{enumerate}
\end{lemma}
\begin{proof} (i) We prove by induction on $n\geq 0$ that for all $A,B\in\F_n$, one has $A\cap B\in\F$. The claim is true for $n=0$ by Lemma \ref{2.2}. Thus let us assume that the claim is true for $n\geq 0$, and let $A,B\in\F_{n+1}$. Then there exist $A_1,A_2,B_1,B_2\in\F_n$ such that $A=A_1\setminus A_2$ and $B=B_1\setminus B_2$. Then by induction hypothesis, there exists $m\geq n$ such that $A_1\cap B_1\in\F_m$. As
\[
A\cap B=(A_1\cap A_2^c)\cap(B_1\cap B_2^c)=((A_1\cap B_1)\setminus A_2)\setminus B_2,
\]
this shows that $A\cap B\in\F_{m+2}\subset\F$ since $(A_1\cap B_1)\setminus A_2\in F_{m+1}$.\\
Assertion (ii) follows readily from the definitions, and
(iii) is established by induction on $n$.
\end{proof}

\par\vspace{2mm}

The next step consists in defining a suitable measure $\mu$ on the $\sigma$-algebra $\sigma(\F)=\sigma(\F_0)$ generated by $\F$ (or equivalently by $\F_0$).

\par\vspace{2mm}
In order to do that, we associate to every element $B=\prod_n B_n\in \F_{0,+}$ the probability measure $\Pp_B$ on the $\sigma$-algebra $\sigma(\mathcal{C})$ generated by the family $\mathcal C$ of all cylinder sets in $X=\prod_n S$. We observe for future use that $\F_0\subset\sigma(\mathcal{C})$, hence that $\sigma(\F)\subset\sigma(\mathcal{C})$ as well.

Then $\Pp_B$ is the product probability measure $\bigotimes_{n}\nu_{n,B}$ where $\nu_{n,B}$ is the probability measure on $\B_S$ given by
\[
\nu_{n,B}(E):=\frac{\nu(E\cap B_n)}{\nu(B_n)}=\nu(E|B_n)
\]
for every $E\in \B_S$ and for every $n$. As is well known, if $C=\prod_nC_n$ with $C_n\subset S$ Borel for every $n$, then
$\Pp_B(C)=\prod_{n=1}^\infty \nu_{n,B}(C_n)$ because $C=\bigcap_{N}C^{(N)}$ where $C^{(N)}=C_1\times C_2\times\cdots\times C_N\times S\times S\times\cdots\in\mathcal C$ and $\Pp_B(C^{(N)})=\prod_{n=1}^N\nu_{n,B}(C_n)$
for all $N$.

\par\vspace{2mm}
We define now a premeasure $\mu$ on $\F$.

\begin{definition}\label{2.4}
For $A\in\F_0$, $A=\prod_nA_n$, set $\mu(A):=\prod_{n=1}^\infty 2\nu(A_n)$. For $A\in\bigcup_{n\geq 1}\F_n$, let $B\in\F_0$ be such that $A\subset B$; then set 
\[
\mu(A)=
\begin{cases}
\Pp_B(A)\mu(B) & \text{if}\ B\in \F_{0,+}\\
0 & \text{if}\ \mu(B)=0.
\end{cases}
\]
\end{definition}

We need to check that $\mu$ is well defined.

\begin{lemma}\label{2.5}
Let $A\in\F$. 
\begin{enumerate}
\item [(i)] If there exists $B\in\F_0$ such that $A\subset B$ and $\mu(B)=0$, then $\Pp_C(A)\mu(C)=0$ for every $C\in\F_{0,+}$ such that $A\subset C$.
\item [(ii)] If $B,C\in\F_{0,+}$ are such that $A\subset B\cap C$, one has
\begin{equation}\label{2.1}
\Pp_B(A)\mu(B)=\Pp_C(A)\mu(C).
\end{equation}
\end{enumerate}
\end{lemma}
\begin{proof} (i) If $B$ and $C$ are as above, then 
\begin{align*}
0\leq \Pp_C(A)\mu(C)
&=\Pp_C(A\cap B)\mu(C)\\
& \leq
\Pp_C(B\cap C)\mu(C)\\
&=
\prod_{n=1}^\infty \frac{\nu(B_n\cap C_n)}{\nu(C_n)}\cdot\prod_{n=1}^\infty 2\nu(C_n)\\
&=
\lim_{N\to\infty}\prod_{n=1}^N \frac{2\nu(B_n\cap C_n)}{2\nu(C_n)}\cdot 2\nu(C_n)\\
&=
\prod_{n=1}^\infty 2\nu(B_n\cap C_n)\leq \prod_{n=1}^\infty 2\nu(B_n)=0.
\end{align*}
(ii) Observe first that if $B,C\in\F_{0,+}$ are such that $\mu(B)=\mu(C)=0$, then Equality (\ref{2.1}) holds trivially. 
By (i), it also holds if $\mu(B)\mu(C)=0$.
Thus, it remains to prove that (\ref{2.1}) holds when $A\subset B\cap C$ with $B,C\in\F_{0,+}$ and $\mu(B)\mu(C)>0$. 

We assume first that $A=\prod_nA_n\in\F_0$; then
\begin{align*}
\Pp_B(A)\mu(B)
&=
\prod_{n=1}^\infty \frac{\nu(A_n\cap B_n)}{\nu(B_n)}\cdot\prod_{n=1}^\infty 2\nu(B_n)\\
&=
\lim_{N\to\infty}\prod_{n=1}^N \frac{2\nu(A_n\cap B_n)}{2\nu(B_n)}2\nu(B_n)\\
&=
\prod_{n=1}^\infty 2\nu(\underbrace{A_n\cap B_n}_{=A_n})\\
&=
\mu(A).
\end{align*}
Similarly, we get $\Pp_C(A)\mu(C)=\mu(A)$.

In the last part of the proof, we fix $B,C\in\F_{0,+}$ such that $\mu(B)\mu(C)>0$, and we define two measures $\mu^{(B)}$ and $\mu^{(C)}$ on the $\sigma$-algebra $\sigma(\mathcal{C})$ (which contains $\sigma(\F)$) by
\[
\mu^{(B)}(E):=\Pp_B(B\cap E)\mu(B)\quad \forall \ E\in\sigma(\mathcal{C})
\]
and similarly for $\mu^{(C)}$. Then $\mu^{(B)}(X)=\Pp_B(B)\mu(B)=\mu(B)<\infty$ (resp. $\mu^{(C)}(X)=\mu(C)$) so that they are both finite measures on $\sigma(\mathcal{C})$. 

Set $\mathcal{A}:=\{A\in\sigma(\F): \mu^{(B)}(A\cap B\cap C)=\mu^{(C)}(A\cap B\cap C)\}$. Then the second part shows that $\F_0\subset \mathcal A$, and in particular, since $B\cap C\in\F_0$, one has that $\mu^{(B)}(B\cap C)=\mu^{(C)}(B\cap C)$, which implies that $X\in\mathcal A$.

Let us check that if $A\in\mathcal A$, then $A^c\in\mathcal A$: indeed, as $B\cap C=(A^c\cap B\cap C)\sqcup(A\cap B\cap C)$,
\begin{align*}
\mu^{(B)}(A^c\cap B\cap C)
&=
\mu^{(B)}(B\cap C)-\mu^{(B)}(A\cap B\cap C)\\
&=
\mu^{(C)}(B\cap C)-\mu^{(C)}(A\cap B\cap C)\\
&=
\mu^{(C)}(A^c\cap B\cap C).
\end{align*}
Finally, it is straightforward to check that $\mathcal A$ is a monotone class. It implies that it is a $\sigma$-algebra which contains $\F_0$, hence $\mathcal A=\sigma(\F)$.
\end{proof}

\par\vspace{2mm}
As a consequence of Lemma \ref{2.5} and Caratheodory's theorem, we have:

\begin{proposition}\label{2.6}
The premeasure $\mu:\F\rightarrow \R_+$ is $\sigma$-additive and thus it extends to a measure still denoted by $\mu$ on the $\sigma$-algebra $\sigma(\F)$. However, $\mu$ is not $\sigma$-finite; in particular, it is infinite.
\end{proposition}
\begin{proof} Let $(A^{(k)})_{k\geq 1}\subset \F$ be a sequence of pairwise disjoint sets such that $A:=\bigsqcup_{k\geq 1}A^{(k)}$ still belongs to $\F$. Choose $B\in\F_0$ such that $A\subset B$. Then, as $A^{(k)}\subset B$ for every $k$,  if $\mu(B)=0$, one has, by Lemma \ref{2.5}, $\mu(A^{(k)})=\mu(A)=0$ for every $k$. If $\mu(B)>0$, one has
\[
\mu(A)=\Pp_B(A)\mu(B)=\sum_k\Pp_B(A^{(k)})\mu(B)=\sum_k\mu(A^{(k)})
\]
since $\Pp_B$ is $\sigma$-additive. Thus, by Caratheodory's theorem (see for instance Theorem 11.1 of \cite{Bil}), $\mu$ extends to a measure still denoted by $\mu$ on $\sigma(\F)$. 

In order to prove that $\mu$ is not $\sigma$-finite, let us choose $A\in \B_S$ such that $\nu(A)=1/2$, and set $A_0=A$ and $A_1=A^c$. Next, for every sequence $\varepsilon:=(\ep_n)_{n\geq 1}\subset \{0,1\}^{\N^*}=:\mathcal E$, set $A_\ep:=\prod_n A_{\ep_n}\in\F_0$. Then $\mu(A_\ep)=1$ for every $\ep\in\mathcal E$ and $A_\ep\cap A_{\ep'}=\emptyset$ for all $\ep\not=\ep'$. If $\mu$ was $\sigma$-finite, then $\mathcal E$ would be countable: indeed, there would exist an increasing sequence $X_1\subset X_2\subset\ldots\subset X$ such that $\bigcup_k X_k=X$ and $\mu(X_k)<\infty$ for every $k$.
Then, for $k\geq 2$, set $\mathcal{E}_k=\{\ep\in\mathcal E : \mu(X_k\cap A_\ep)\geq 1/2\}$. Then $|\mathcal{E}_k|\leq 2\mu(X_k)$ is finite. As $\lim_{k\to\infty}\mu(X_k\cap A_\ep)=\mu(A_\ep)=1$ for every $\ep$, it follows that $\mathcal E=\bigcup_k\mathcal{E}_k$ would be countable, which is not the case.
\end{proof}

\par\vspace{2mm}
Now we consider the diagonal action of $G$ on $(X,\sigma(\F),\mu)$, \textit{i.e.} 
\[
g\cdot(s_n)_{n\geq 1}=(gs_n)_{n\geq 1}
\] 
for $g\in G$ and $(s_n)_{n\geq 1}\in X=\prod_n S$.


\begin{lemma}\label{2.7}
The action of $G$ on $X$ defined above is measurable. In particular, $G$ acts by measurable automorphisms on $(X,\sigma(\mathcal{F}),\mu)$ and it preserves $\mu$.
\end{lemma}
\begin{proof}
Let $\alpha : G\times X\rightarrow X$ be defined by $\alpha(g,(x_{n})_{n \ge 1})=(gx_{n})_{n \ge 1}$.
To prove that $\alpha$ is measurable, it is sufficient to see that the preimage of any element of $\mathcal{F}_0$ is measurable.\\
Let $A\in \mathcal{F}_0$ with $\displaystyle A=\prod_{n \ge 1} A_{n}$. For $m\ge 1$, we define a permutation function $\beta_{m} : G\times X \rightarrow G \times X$ by 
\[ 
(g, (x_{n})_{n\ge 1})\mapsto (g,(x_{m},x_{1}, x_{2}, \ldots, x_{m-1}, x_{m+1}, x_{m+2}, \cdots), 
\]
which is clearly measurable.\\
Let us consider now $\alpha^{-1}(A)$; we have:
\begin{align*} 
\alpha^{-1}(A)
&=\left\{(g, (x_{n})_{n\ge 1}):\, \alpha((g, (x_{n})_{n\ge 1}))\in A \right\}\\
&= \left\{(g, (x_{n})_{n\ge 1}):\, (g x_{n})_{n\ge 1}\in A \right\}\\
&=\displaystyle \bigcap_{k \ge 1} \left\{(g, (x_{n})_{n\ge 1}):\, gx_{k}\in A_{k} \right\}.
\end{align*}
For fixed $k$, since the action $\gamma : G\times S \rightarrow S$ given by $\gamma(g,s)=gs$ is continuous hence measurable, we have 
\[  
\left\{(g, (x_{n})_{n\ge 1}):\, gx_{k}\in A_{k} \right\}= \beta_{k}^{-1}\left(\gamma^{-1}(A_{k})\times S \times S \times \ldots\right) \in \mathcal{B}(G)\times \sigma(\mathcal{F}). 
\]
Whence 
\[ 
\alpha^{-1}(A)= \bigcap_{n \ge 1} \beta_{n}^{-1}\left(\gamma^{-1}(A_{n})\times \prod_{k \ge 2} S \right) \in \mathcal{B}(G)\times \sigma(\mathcal{F}).
\]
In particular, if $B\in \sigma(\mathcal{F})$ and for $g\in G$ fixed, then \[gB=(g^{-1})^{-1}B\in \sigma(\mathcal{F})\]
This ends the proof of the measurability of the action of $G$.

Finally, if $A=\prod_n A_n\in\F_0$, then $gA=\prod_n gA_n$, and the equality $\mu(gA)=\mu(A)$ holds since the action of $G$ is diagonal and $\nu$ is preserved by the action of $G$ on $S$. 

If $A\in\F$, let $B\in\F_{0,+}$ be such that $A\subset B$, so that $\mu(A)=\Pp_B(A)\mu(B)$, according to Definition \ref{2.4}. Then $gA\subset gB$ so that
\[
\mu(gA)=\Pp_{gB}(gA)\mu(gB)=\Pp_B(A)\mu(B)=\mu(A)
\] 
for the following reason: For every cylinder set $C=C_1\times\cdots\times C_N\times S\times\cdots\in\sigma(\mathcal C)$, one has
\[
\Pp_{gB}(gC)=\prod_{n=1}^N \frac{\nu(gC_n\cap gB_n)}{\nu(gB_n)}=\Pp_B(C).
\]
This equality holds first for every element of the algebra generated by cylinder sets, hence for every element $C\in\sigma(\mathcal C)$ by uniqueness of probability measures which coincide on given algebras of sets.
In particular, one has $\mu(gA)=\mu(A)$ for every $A\in\F$, hence in the algebra $\Rr(\F)$.

Next, the construction of the extension of $\mu$ to $\sigma(\F)$ is given by
\[
\mu(E)=\inf\left\{\sum_{m\ge 1} \mu(B_m) : (B_m)_{m\ge 1}\subset\mathcal{R}(\F), E\subset \bigcup_m B_m\right\}
\]
for every $E\in\sigma(\F)$. Thus, if $g\in G$ is fixed, one has
\[
\mu(gE)=\inf\left\{\sum_{m\ge 1} \mu(B_m) : (B_m)_{m\ge 1}\subset\mathcal{R}(\F), E\subset \bigcup_m g^{-1}B_m\right\}
=\mu(E)
\]
because the ring $\mathcal{R}(\F)$ is $G$-invariant and every countable covering of $E\in\sigma(\F)$ in the definition of $\mu(E)$ above can be taken of the form
\[
E\subset \bigcup_m g^{-1}B_m
\]
with $(B_m)\subset \mathcal{R}(\F)$ by $G$-invariance of $\mathcal{R}(\F)$.
\end{proof}

\par\vspace{2mm}
We are now ready to prove the first part of Theorem \ref{thmC0} in the general case, namely the existence of a $C_0$-dynamical system with almost invariant vectors, and to finish the proof in the case where $G$ is discrete (hence countable).

\begin{proposition}\label{2.8}
The dynamical system $(X,\sigma(\F),\mu,G)$ is a $C_0$-dynamical system, namely, for all $A,B\in\sigma(\F)$ such that $0<\mu(A),\mu(B)<\infty$ one has
\begin{equation}\label{C0}
\lim_{g\to\infty}\mu(gA\cap B)=0,
\end{equation} and there exists a sequence of unit vectors $(\xi_n)\in L^2(X,\sigma(\F),\mu)$ such that $\xi_n\ge 0$ for every $n$, and for every compact set $K\subset G$ one has
\begin{equation}\label{presquinv}
\lim_{n\to\infty}\sup_{g\in K}\langle \pi_X(g)\xi_n|\xi_n\rangle=1.
\end{equation}
Moreover, if $G$ is discrete, there exists a $G$-invariant subset $\Om\in\sigma(\F)$ such that the restriction of $\mu$ to $\Om$ is $\sigma$-finite and the corresponding Hilbert space $L^2(\Om,\mu)$ contains the sequence $(\xi_n)$ and is separable.
\end{proposition}
\begin{proof}
Assume first that $A,B\in\F_0$, and write $A=\prod_nA_n$ and $B=\prod_nB_n$, so that 
\[
\mu(A)=\prod_{n \ge 1} 2 \nu(A_n) \quad \text{and} \quad \mu(B)=\prod_{n \ge 1} 2 \nu(B_{n}).
\]
Let $\ep>0$ be fixed and take  $\ep'>0$ small enough in order that $\delta :=\frac{1}{2}+ \ep' +\frac{\ep'}{\frac{1}{2}- \ep'}<1$.

Since $0<\mu(A),\mu(B)<\infty$, there exists $N$ large enough such that 
\[
\frac{1}{2}-\ep'<\nu(A_{n}),\nu(B_{n})<\frac{1}{2}+\ep' \quad \forall n\ge N.
\]
Since $\delta < 1$, there exists $m$ large enough such that $\delta^{m+1}<\frac{\ep}{\mu(A)}$. The action of $G$ on $(S,\nu)$ being strongly mixing, there exist compact sets $K_{n}\subset G$ for all $n\in \left\{N,\ldots, N+m\right\}$ such that
\[
|\nu(gA_{n}\cap B_{n})-\nu(A_{n})\nu(B_{n})|\le \ep' \quad \forall g\in G\setminus K_{n}.
\]
Set $\displaystyle K=\bigcup_{n=N}^{N+m} K_{n}$ which is compact. Then we have for all $g\in G\setminus K$:
\begin{align*}
\mu(gA\cap B)
&\leq 
\prod_{n=1}^{N-1} 2\nu(A_{n})\cdot\prod_{n=N}^{N+m}2(\nu(A_{n})\nu(B_{n})+\ep')\cdot\prod_{n\ge N+m+1}2\nu(A_{n})\\
&= 
\mu(A) \cdot \prod_{n=N}^{N+m} \frac{2\nu(A_{n})\nu(B_{n})+2\ep'}{2\nu(A_{n})}\\
&=
\mu(A)\prod_{n=N}^{N+m} \left(\nu(B_{n})+\frac{\ep'}{\nu(A_{n})}\right)\\
&< 
\mu(A)\prod_{n=N}^{N+m} \left(\frac{1}{2}+ \ep' +\frac{\ep'}{\frac{1}{2}- \ep'}\right)\\
&=
\mu(A)\delta^{m+1}<\ep.
\end{align*}
Thus we have 
\[
\lim_{g\rightarrow \infty} \mu(gA\cap B)=0 \quad \forall A,B\in \mathcal{F}_0.
\]

The same claim holds for $A,B\in\F$ since there exist $C,D\in\F_0$ such that $A\subset C$ and $B\subset D$ and $\mu(gA\cap B)\le \mu(gC\cap D)\to 0$ as $g\to\infty$. Moreover, Equality (\ref{C0}) also holds for all elements of the ring $\Rr(\F)$. 

Finally, if $A,B\in\sigma(\F)$ are such that $0<\mu(A),\mu(B)<\infty$, by construction of the measure $\mu$ on $\sigma(\F)$, if $\ep>0$ is given, there exist two sequences $(C_k)_{k\geq 1},(D_\ell)_{\ell\ge 1}\subset \Rr(\F)$ such that 
\[
A\subset \bigcup_{k\geq 1}C_k\quad\text{and}\quad B\subset\bigcup_{\ell\ge 1}D_\ell
\]
and
\[
\mu(A)\leq \sum_k\mu(C_k)<\mu(A)+\ep\quad\text{and}\quad \mu(B)\le\sum_\ell \mu(D_\ell)<\mu(B)+\ep.
\]
Choose first $N$ large enough so that $\sum_{\ell>N}\mu(D_\ell)<\ep/3$. Then, as
\[
gA\cap B\subset \left(\bigcup_{\ell=1}^N gA\cap D_\ell\right) \cup \left(\bigcup_{\ell>N}D_\ell\right),
\]
we get
\[
\mu(gA\cap B)\le \sum_{\ell=1}^N\mu(gA\cap D_\ell)+\ep/3.
\]
Choose next $M$ large enough so that $\sum_{k>M}\mu(C_k)<\ep/3N$. Then, as
\[
gA\cap D_\ell\subset \left(\bigcup_{k=1}^M gC_k\cap D_\ell\right)\cup\left(\bigcup_{k>M}gC_k\right)
\]
for every $1\leq\ell\leq N$, and since $\mu$ is $G$-invariant, we get
\[
\mu(gA\cap B)\leq \sum_{\ell=1}^N\sum_{k=1}^M\mu(gC_k\cap D_\ell)+2\ep/3
\]
for every $g\in G$. By the previous part of the proof, there exists a compact set $K\subset G$ such that 
\[
\mu(gA\cap B)<\ep\quad \forall g\in G\setminus K.
\]
This ends the proof of the first claim of the proposition.

Let us prove now the existence of the sequence $(\xi_n)\subset L^2(X,\sigma(\F),\mu)$ which satisfies \eqref{presquinv}. 

The probability standard space $(S,\nu)$ contains an asymptotically invariant sequence $(A_{n})_{n \ge 1 } \subset \mathcal{B}_S$ such that $\nu(A_{n})=\frac{1}{2}$ for all $n$, and for every compact set $K \subset G$,
\begin{equation}
\label{lemme4.2}
\sup_{g\in K }\nu(gA_{n} \cap A_{n} )\rightarrow \frac{1}{2} \quad \text{as} \quad n\rightarrow \infty.
\end{equation}
Since $G$ is a locally compact, second countable group, we choose an increasing sequence of compact sets $(K_{n})_{n\ge 1}\subset G$ such that $\displaystyle G=\bigcup_{n\ge 1}K_{n}$ and such that for every compact set $K\subset G$, there exists $m\ge 1$ such that $K\subset K_m$. Consequently, by \eqref{lemme4.2}, for all $m,k\ge 1$ there exists an integer $n(k,m)$ such that 
\[
\left|\nu(gA_{n(k,m)} \cap A_{n(k,m)}) - \frac{1}{2}\right|\le \frac{1}{2}\left(1-e^{-\frac{1}{m2^{k}}}\right)\quad \forall g\in K_{m}.
\]
Then we set for all $m$ 
\[
\xi_{m}=\chi_{\prod_{k \ge 1} A_{n(k,m)}} \in  L^{2}(X,\mu).
\]
By construction, $0\leq \xi_m\leq 1$ and $\|\xi_{m}\|_{2}=1$ for all $m$.\\
Now let $K$ be a compact subset of $G$; then for every integer $m\ge 1$ such that $K\subset K_{m}$, 
we have:
\begin{align*}
1\geq \langle \pi_{X}(g)\xi_{m}| \xi_{m} \rangle 
&= 
\int_{X} \chi_{\prod_{k \ge 1} A_{n(k,m)}}(g^{-1}x)\chi_{\prod_{k \ge 1} A_{n(k,m)}}(x)d\mu(x)\\
&= 
\int_{X} \chi_{\prod_{k \ge 1}\left(gA_{n(k,m)}\cap A_{n(k,m)}\right)}(x)d\mu(x)\\
&=\mu\left(\prod_{k \ge 1}\left(gA_{n(k,m)}\cap A_{n(k,m)}\right)\right)=\prod_{k\ge 1} 2\nu(gA_{n(k,m)}\cap A_{n(k,m)})\\
&\ge 
\prod_{k \ge 1} 2\left(\frac{1}{2} - \frac{1}{2}(1-e^{-\frac{1}{m2^{k}}})\right)=\prod_{k\ge 1} e^{-\frac{1}{m2^{k}}}\\
&=
e^{-\sum_{k=1}^{\infty}\frac{1}{m2^{k}}}=e^{-\frac{1}{m}} \rightarrow_{m \rightarrow \infty} 1,
\end{align*}
uniformly on $K$, where the first inequality follows from the Cauchy-Schwarz Inequality.

Assume now that $G$ is discrete, hence countable; set
\[
\Omega=\bigcup_{g\in G} \bigcup_{m\ge 1}g\cdot X_m
\]
where $X_m=\prod_{k}A_{n(k,m)}$ is the support of $\xi_m$ for every $m$. Then $\Om$ is $G$-invariant, the restriction of $\mu$ to $\Om$ is $\sigma$-finite and $L^2(\Om,\mu)$ is separable since $G$ is countable, and it contains the sequence $(\xi_m)$ by construction.
\end{proof}

\par\vspace{2mm}

Thus, the proof of Theorem \ref{thmC0} is complete in the case where $G$ is discrete, hence we will focus on the case where $G$ is not discrete.

\begin{remark}\label{nonContinue}
We are very grateful to Alessandro Calderi and Alain Valette for having kindly communicated the present remark. The associated representation $\pi_X$ of $G$ on $L^2(X,\sigma(\F),\mu)$ is not continuous in general. Indeed, if it was continuous, then we would have
\[
\lim_{g\rightarrow e}\langle\pi_X(g)\xi|\xi\rangle =\|\xi\|^2_2
\]
for every element $\xi\in L^2(X,\mu)$, and in particular $\lim_{g\rightarrow e}\mu(gB\cap B)=1$ for every $B\in\F_0$ with $\mu(B)=1$. Choose a Borel set $A\subset S$ such that $\nu(A)=1/2$, and set $B=\prod_nA\in\F_0$. If $g\in G$ is such that 
$\nu(gA\bigtriangleup A)>0$, we have $\mu(gB\cap B)=0$ since $\nu(gA\cap A)<1/2$. If we can make $g\to e$, we have proved that $\pi_X$ is not continuous. It is the case for $G=S=S^1$ equipped with its normalized Lebesgue measure.

We also observe that, in the case where $G$ is countable, provided that we put $X_0=B=\prod_n A$ as above and we take $\Omega=\bigcup_{g\in G} \bigcup_{m\ge 0}g\cdot X_m$, the unit vector $\chi_B\in L^2(\Omega,\mu)$ satisfies the following condition:
\[
\f_B(g):=\langle\pi_X(g)\chi_B|\chi_B\rangle=
\left\{
\begin{array}{cc}
1 & g=e\\
0 & g\neq e.
\end{array}\right.
\]
This means that $\f_B$ is the positive definite function $\delta_e$ whose GNS construction is the (left) regular representation of $G$. Hence, the left regular representation of $G$ is a subrepresentation of $\pi_\Omega$.

If $G$ is not discrete, we see no reason why $\pi_\Om$ should contain the regular representation, as in this case there is no analogue of the function $\delta_e$.
\end{remark}

\section{Proof of Theorem \ref{thmC0}, Part 2}

Before proving the last part of our main theorem, let us make the following comment which we owe to S. Baaj: Recall that if $G$ be a lcsc group that acts measurably on some measure space $(\Om,\mu)$ where $\mu$ is $G$-invariant and $\sigma$-finite, and if $L^2(\Om,\mu)$ is separable, then the associated unitary representation $\pi_\Om$ is automatically continuous. This follows from Proposition A.6.1 of \cite{BHV} for instance. 

This means that, if we had been able to restrict our action of $G$ to such a measure space, then the proof of Theorem \ref{thmC0} would be complete. Unfortunately, we were unable to do that, and thus we have to proceed as follows: we define a subfamily $\F_c$ of $\F$ and a measure $\mu_c$ on $\sigma(\F_c)$ so that the permutation representation on $L^2(X,\sigma(\F_c),\mu_c)$ is continuous, and then, using continuity, we are able to restrict the action of $G$ to a $\sigma$-finite measure subspace of $X$ which has all the desired properties.

\par\vspace{2mm}
Before defining the above mentionned family $\F_c$, we fix an increasing sequence of compact sets $(K_n)_{n\ge 1}$ of $G$ with the following properties: $e\in \mathring{K}_1$, $K_n\subset \mathring{K}_{n+1}$ for every $n\ge 1$ and $G=\bigcup_{n\geq 1}K_n$.
We also set $K:=K_1$ which is a compact neighbourhood of $e$.

\begin{definition}
\begin{enumerate}
\item [(1)] The family $\F_{c,0}$ is the collection of sets $A=\prod_nA_n\in\F_0$ such that $\mu(A)=\prod_n2\nu(A_n)=0$ or 
that 
\[
\lim_{N\to\infty}\prod_{n=N}^\infty 2\nu(gA_n\cap A_n)=1
\]
uniformly for $g\in K$.
\item [(2)] We define the sequence $(\F_{c,n})_{n\ge 1}$ of collections of subsets of $X$ by induction: $\F_{c,n}:=\{B\setminus A : A,B\in\F_{c,n-1}\}$, and we set finally $\F_c=\bigcup_{n\ge 0}\F_{c,n}$.
\end{enumerate}
 We also denote by $\sigma(\F_c)$ the $\sigma$-algebra generated by $\F_c$; it is a sub-$\sigma$-algebra of $\sigma(\F)$.
\end{definition}

We observe that the sequence $(\prod_{k\ge 1}A_{n(k,m)})_{m\ge 1}$ constructed in the proof of Proposition \ref{2.8} is contained in $\F_{c,0}$. Indeed, we have proved that 
\[  
2\nu(gA_{n(k,m)} \cap A_{n(k,m)}) \in \left[e^{-\frac{1}{m2^{k}}}, 1 \right] \quad \forall g\in K. 
\]
Hence, we get
\begin{align*}
1
&\ge 
\lim_{N\rightarrow \infty } \prod_{k=N}^{\infty} 2\nu(gA_{n(k,m)} \cap A_{n(k,m)}) \\
&\ge 
\lim_{N\rightarrow \infty } \prod_{k=N}^{\infty} e^{-\frac{1}{m2^{k}}} \\
&= 
e^{-\frac{1}{m} \lim_{N\rightarrow \infty } \sum_{k=N}^{\infty} \frac{1}{2^{k}}}=1
\end{align*}
uniformly on $K$. This proves among others that the associated sequence of vectors $(\xi_m)_{m\ge 1}\subset L^2(X,\sigma(\F_c),\mu)$ is an almost invariant sequence of unit vectors.

\par\vspace{2mm}
In order to see that $\F_{c,0}$ is stable under finite intersections, we need the following elementary lemma.

\begin{lemma}\label{convergence}
Let $(x_k)_{k\ge 1},(a_k)_{k\geq 1}\subset\R^+$ be sequences such that
\[
\lim_{N\to\infty}\prod_{k\ge N} x_{k}=\lim_{N\to\infty}\prod_{k\ge N} (1-a_{k})=1.
\]
Then
\[
\lim_{N\to\infty}\prod_{k\ge N} (x_{k}-a_{k})=1.
\]
\end{lemma}
\begin{proof} There exists $N_0$ so that $x_k>2/3$ and $0\leq a_k<1/2$ for all $k\geq N_0$. This implies immediately that
\[
(-2a_{k} +a_{k}^{2})x_{k} \le -\frac{3}{2}a_{k}x_{k}\le -a_{k}
\]
for all $k\geq N_0$, and we get for all $N\ge N_0$
\[
\prod_{k \ge N} x_{k}(1-a_{k})^{2} \le \prod_{k \ge N} (x_{k}-a_{k})\le \prod_{k \ge N} x_{k}
\]
which proves the claim.
\end{proof}

\begin{lemma}\label{intersection}
Let $A,B\in\F_{c,0}$. Then $A\cap B\in\F_{c,0}$.
\end{lemma}
\begin{proof} Let $A,B \in \F_{c,0}$ where $A=\prod_{n} A_{n}$ and $B=\prod_{n} B_{n}$.\\
If $A$ or $B$ has measure zero, then $\mu(A\cap B)=0$ and $A\cap B \in \mathcal{F}_{c,0}$. We assume that $\mu(A \cap B)\neq 0$. Thus, we have that the five sequences
\begin{equation*}
\prod_{n=N}^{\infty} 2 \nu(gA_{n} \cap A_{n}), \prod_{n=N}^{\infty} 2 \nu(gB_{n} \cap B_{n}), \prod_{n=N}^{\infty} 2 \nu(A_{n} \cap B_{n}), \prod_{n=N}^{\infty} 2 \nu(A_{n}), \prod_{n=N}^{\infty} 2 \nu(B_{n})
\end{equation*}
converge to $1$ uniformly on $K$ as $N\to\infty$.\\
As $\nu(A_{n}),\nu(B_{n})\neq 0$ for every $n$, we have
\begin{equation*}
\prod_{n=N}^{\infty} \frac{2 \nu(gA_{n} \cap A_{n})}{2\nu(A_{n})}, \prod_{n=N}^{\infty} \frac{2 \nu(gB_{n} \cap B_{n})}{2\nu(B_{n})}, \prod_{n=N}^{\infty} \frac{2 \nu(A_{n} \cap B_{n})}{2\nu(A_{n})} \text{ and } \prod_{n=N}^{\infty} \frac{2 \nu(A_{n} \cap B_{n})}{2\nu(B_{n})}
\end{equation*}
converge to $1$ uniformly on $K$ as $N\to\infty$.\\
Moreover,
\begin{align*}
1 &\ge 
\lim_{N \rightarrow \infty} \prod_{n=N}^{\infty} \frac{2 \nu(g(A_{n}\cap B_{n}) \cap (A_{n}\cap B_{n}))}{2\nu(A_{n}\cap B_{n})} \\
&\ge 
\lim_{N \rightarrow \infty} \prod_{n=N}^{\infty} \frac{2 \nu(A_{n}\cap B_{n}) - 2 \nu(A_{n}\cap B_{n}\cap gA_{n}^{c})-2\nu(A_{n}\cap B_{n} \cap gA_{n} \cap gB_{n}^{c})}{2\nu(A_{n}\cap B_{n})}\\
& \ge 
\lim_{N \rightarrow \infty} \prod_{n=N}^{\infty} \left(1- \frac{2\nu(A_{n}\setminus gA_{n})}{2 \nu(A_{n}\cap B_{n})} - \frac{2 \nu(B_{n}\setminus gB_{n})}{2 \nu(A_{n}\cap B_{n})}\right)
\end{align*}
Using Lemma \ref{convergence}, we will show that 
\begin{equation}\label{eq}
\lim_{N \rightarrow \infty} \prod_{n=N}^{\infty} \left(1- \frac{2\nu(A_{n}\setminus gA_{n})}{2 \nu(A_{n}\cap B_{n})} - \frac{2 \nu(B_{n}\setminus gB_{n})}{2 \nu(A_{n}\cap B_{n})}\right)=1.
\tag{$\ast$}
\end{equation}
One has to check that
\begin{enumerate}
\item $\lim_{N \rightarrow \infty} \prod_{n=N}^{\infty}  \left(1- \frac{2 \nu(B_{n}\setminus gB_{n})}{2 \nu(A_{n}\cap B_{n})}\right) =1;$
\item $\lim_{N \rightarrow \infty} \prod_{n=N}^{\infty}  \left(1- \frac{2\nu(A_{n}\setminus gA_{n})}{2 \nu(A_{n}\cap B_{n})}\right) =1.$
\end{enumerate}
We prove (1) in detail; as the proof of (2) is similar, we will get $(\ast)$. One has
\begin{align*}
\lim_{N \rightarrow \infty} 
&
\prod_{n=N}^{\infty}  \left(1- \frac{2 \nu(B_{n}\setminus gB_{n})}{2 \nu(A_{n}\cap B_{n})}\right)\\
&= 
\lim_{N \rightarrow \infty} \prod_{n=N}^{\infty}  \left(1- \frac{2 \nu(B_{n}\setminus gB_{n})}{2 \nu(A_{n}\cap B_{n})}\right) \lim_{N \rightarrow \infty} \prod_{n=N}^{\infty} \frac{2 \nu(A_{n} \cap B_{n})}{2\nu(B_{n})} \\
&= 
\lim_{N \rightarrow \infty} \prod_{n=N}^{\infty} \left(\frac{2\nu(A_{n}\cap B_{n})- 2\nu(B_{n}\setminus gB_{n})}{2\nu(B_{n})}\right). \\
\end{align*}
Moreover,
\begin{align*}
\lim_{N \rightarrow \infty} 
&
\prod_{n=N}^{\infty} \left(1- \frac{2\nu(B_{n}\setminus gB_{n})}{2\nu(B_{n})}\right)\\
&=
\lim_{N \rightarrow \infty} \prod_{n=N}^{\infty} \left(\frac{2\nu(B_{n})- 2\nu(B_{n}\setminus gB_{n})}{2\nu(B_{n})}\right)\\
&=
\lim_{N \rightarrow \infty} \prod_{n=N}^{\infty} \frac{2\nu(B_{n}\cap gB_{n})}{2\nu(B_{n})}=1.
\end{align*}
As observed above, this ends the proof.
\end{proof}

\par\vspace{2mm}
Next, exactly the same arguments as those in Section 2 from Lemma 2.3 to Proposition 2.8 allow to prove the following facts:

\begin{proposition}\label{resume}
The family $\F_c$ has the following properties:
\begin{enumerate}
\item [(i)] For all $A,B\in\F_c$, one has $A\cap B\in\F_c$ and $A\setminus B\in\F_c$, so that $\F_c$ is a semiring, and the set of all finite disjoint unions of elements of $\F_c$ is the ring generated by $\F_c$ which is denoted by $\Rr(\F_c)$.
\item [(ii)] The diagonal action of $G$ on $X=\prod_{n\geq 1}S$ is $\sigma(\F_c)$-measurable.
\item [(iii)] There exists a measure $\mu_c:\sigma(\F_c)\rightarrow [0,\infty]$ which is $G$-invariant and such that $\mu_c(\prod_n A_n)=\prod_n 2\nu(A_n)$ for every $A=\prod_n A_n\in \F_{c,0}$.
\item [(iv)] The dynamical system $(X,\sigma(\F_c),\mu_c,G)$ is a $C_0$-dynamical system.
\item [(v)] Let $(A_{n(k,m)})_{k,m}$ be the family constructed in the proof of Proposition \ref{2.8}, set $X_m=\prod_k A_{n(k,m)}$ and $\xi_m=\chi_{X_m}$ for every $m\ge 1$. Then $(\xi_m)\subset L^2(X,\sigma(\F_c),\mu_c)$ is a sequence of almost invariant unit vectors.
\end{enumerate}
\end{proposition}

We are ready to prove that the representation $\pi_X$ on $L^2(X,\sigma(\F_c),\mu_c)$ is continuous. To do this, it suffices to prove that, for every $A\in\sigma(\F_c)$ such that $\mu_c(A)<\infty$, one has
\[
\Vert \pi_X(g)\chi_A-\chi_A\Vert_2^2=\int_X|\chi_{gA}-\chi_A|^2d\mu_c=
\int_X|\chi_{gA}-\chi_A|d\mu_c=
\mu_c(gA\bigtriangleup A){\to} 0
\]
as $g\to e$.

\begin{proposition}\label{continuite}
Let $A\in\sigma(\F_c)$ be such that $\mu_c(A)<\infty$. Then 
\[
\lim_{g\to e}\mu_c(gA\bigtriangleup A)=0.
\]
\end{proposition}
\begin{proof} Denote by $\Ss$ the family of sets $A\in\sigma(\F_c)$ such that $\mu_c(A)<\infty$ and that $\lim_{g\to e}\mu_c(gA\bigtriangleup A)=0$. 
Let us prove the following assertions:\\
(i) One has $\F_c\subset\Ss$, \textit{i.e.} $\lim_{g\to e}\mu(gA\bigtriangleup A)=0$ for every $A\in\F_c$.\\
Indeed, If $\mu(A)=0$, the claim is obvious. Thus assume that $A=\prod_nA_n$ and
\[
\lim_{N\rightarrow \infty} \prod_{n=N}^{\infty} 2\nu(gA_{n} \cap A_{n})=1
\]
uniformly on $K$.
Let $(g_{m})_{m \ge 1}$ be a sequence in $K$ which converges to $e$. We prove that $\mu_c(g_mA\cap A)\to\mu_c(A)$ as $m\to\infty$; this will prove the claim since we have for every $m$
\begin{align*}
\mu_c(g_mA\bigtriangleup A)
&=
\int_X|\chi_{g_mA}-\chi_A|d\mu_c\\
&\leq 
\int_X(\chi_A-\chi_{g_mA\cap A})d\mu_c+\int_X(\chi_{g_mA}-\chi_{g_mA\cap A})d\mu_c\\
&=
2(\mu_c(A)-\mu_c(g_mA\cap A)).
\end{align*}
(Notice that we have used that $\mu_c(A)<\infty$ and that $\mu_c$ is $G$-invariant.)

As $\prod_{n=N}^{\infty} 2\nu(gA_{n} \cap A_{n})\to 1$ uniformly on $K$, there exists $N_0$ such that 
\[ 
\prod_{n=N_0+1}^{\infty}\frac{ 2\nu(g_{m}A_{n} \cap A_{n})}{2\nu(A_{n})}\in \left(\sqrt{1-\ep},\sqrt{1+\ep}\,\right) \quad \forall m.
\] 
As the representation of $G$ on $L^2(S,\nu)$ is continuous, one has, for every fixed $n$,
\[ 
\nu(g_{m}A_{n}\cap A_{n})\rightarrow\nu(A_{n})\quad \text{as }m\rightarrow \infty.
\]
Hence there exists $M$ such that, for every $n\in \left\{1,\ldots,N_{0} \right\}$,
\[ 
\frac{\nu(g_{m} A_{n} \cap A_{n} )}{\nu(A_{n})} \in \left[\left(1-\ep\right)^\frac{1}{2N_0},\left(1+\ep\right)^\frac{1}{2N_0}\right] \quad \forall m \ge M. 
\]
Then one has for every $m \ge M$ 
\begin{align*}
\frac{\prod_{n \ge 1} 2 \nu(g_{m} A_{n} \cap A_{n})}{\prod_{n\ge 1} 2 \nu(A_{n})} 
&= 
\left(\prod_{n=1}^{N_0}\frac{ 2\nu(g_{m}A_{n} \cap A_{n})}{2\nu(A_{n})}\right)\cdot \left(\prod_{n=N_0+1}^{\infty}\frac{ 2\nu(g_{m}A_{n} \cap A_{n})}{2\nu(A_{n})}\right)\\
&\in (1-\ep,1+\ep).
\end{align*}
This shows that $\frac{\mu_c(g_{m}A\cap A)}{\mu_c(A)} \rightarrow 1$ and $\mu_c(g_{m}A\cap A)\rightarrow \mu_c(A)$ as $m\to\infty$.\\
(ii) If $A,B\in\Ss$, then $A\cap B\in\Ss$ and $A\setminus B\in\Ss$. In particular, $\Ss$ is a semiring of subsets of $X$ which contains $\F_c$.\\
Indeed, let $A,B\in\Ss$ and let $g\in G$. Then
\begin{align*}
\mu_c(g(A\cap B)\bigtriangleup (A\cap B))
&=
\int_X|\chi_{gA}\chi_{gB}-\chi_A\chi_B|d\mu_c\\
&\le
\int_X|\chi_{gA}(\chi_{gB}-\chi_B)|d\mu_c
+\int_X|\chi_B(\chi_{gA}-\chi_A)|d\mu_c\\
&\le 
\mu_c(gA\bigtriangleup A)+\mu_c(gB\bigtriangleup B)\to 0
\end{align*}
as $g\to e$. This shows that $A\cap B\in\Ss$.
Next,
\begin{align*}
\mu_c(g(A\setminus B)\bigtriangleup (A\setminus B))
={}&
\int_X|\chi_{g(A\setminus B)}-\chi_{A\setminus B}|d\mu_c\\
={}&
\int_X|\chi_{gA}(1-\chi_{gB})-\chi_A(1-\chi_B)|d\mu_c\\
={}&
\int_X|\chi_{gA}-\chi_A-(\chi_{gA}\chi_{gB}-\chi_A\chi_B)|d\mu_c\\
\le{}&
\mu_c(gA\bigtriangleup A)+\int_X|\chi_{gA}-\chi_A|\chi_{gB}d\mu_c\\
&+
\int_X\chi_A|\chi_{gB}-\chi_B|d\mu_c\\
\le{}& 
2\mu_c(gA\bigtriangleup A)+\mu_c(gB\bigtriangleup B)\to 0
\end{align*}
as $g\to e$, showing that $A\setminus B\in\Ss$. Lemma \ref{intersection} and these facts imply that $\F_c\subset\Ss$. \\
(iii) Let $A_1,\ldots,A_n\in\Ss$. Then their union $\bigcup_{j=1}^nA_j\in\Ss$. In particular, $\Ss$ is a ring of subsets of $X$ which contains the ring generated by $\F_c$.\\
Indeed, by (ii), we can assume that $A_i\cap A_j=\emptyset$ for all $i\not=j$. Setting $A=\bigsqcup_{j=1}^n A_j$, we have
\begin{align*}
\mu_c(gA\bigtriangleup A)
&=
\int_X\left|\sum_{j=1}^n \chi_{gA_j}-\sum_{j=1}^n\chi_{A_j}\right|d\mu_c\\
&=
\int_X\left|\sum_{j=1}^n (\chi_{gA_j}-\chi_{A_j})\right|d\mu_c\\
&\le
\sum_{j=1}^n \mu_c(gA_j\bigtriangleup A_j)\to 0
\end{align*} 
as $g\to e$. This proves (iii).\\
(iv) Let $A\in\sigma(\F_c)$ be such that $\mu_c(A)<\infty$. Then
\[
\lim_{g\to e}\mu_c(gA\bigtriangleup A)=0.
\]
Indeed, fix $\ep>0$. By construction of the measure $\mu_c$ on $\sigma(\F_c)$, and since $\Ss$ is a ring which contains the ring $\Rr(\F_c)$ generated by $\F_c$,
\[
\mu_c(A)=\inf\left\{\sum_{i\ge 1}\mu_c(A_i): A_i\in\F_c\ \forall i,\ A\subset\bigcup_i A_i\right\}.
\] 
Hence, there exists a sequence $(A_i)_{i\ge 1}\subset \F_c$ such that $A\subset \bigcup_i A_i:=B$ and 
\begin{equation}\label{infimum}
\mu_c(A)\le \sum_i\mu_c(A_i)<\mu_c(A)+\frac{\ep}{5}.
\end{equation}
By (ii), we assume that $A_i\cap A_j=\emptyset$ for all $i\not=j$. Choose then $N$ large enough so that 
\[
\sum_{i>N}\mu_c(A_i)<\frac{\ep}{5},
\]
and let $V\subset G$ be an open neighbourhood of $e$ such that 
\[
\sum_{i=1}^N\mu_c(gA_i\bigtriangleup A_i)<\frac{\ep}{5}\quad \forall g\in V.
\]
Then we have for every $g\in V$
\begin{align*}
\mu_c(gA\bigtriangleup A)
&\le
\int_X|\chi_{gA}-\chi_{gB}|d\mu_c+\int_X|\chi_{gB}-\chi_B|d\mu_c
+\int_X|\chi_B-\chi_A|d\mu_c\\
&\le
2\mu_c(B\setminus A)+\sum_i \mu_c(gA_i\bigtriangleup A_i)\\
&\le
2\mu_c(B\setminus A)+\sum_{i=1}^N\mu_c(gA_i\bigtriangleup A_i)+2\sum_{i>N}\mu_c(A_i)<\ep.
\end{align*}
\end{proof}

\begin{remark}
Since $\mu$ is not $\sigma$-finite, there is no reason that it coincides with $\mu_c$ on $\sigma(\F_c)$, even if it does on $\F_{c}\subset \F$. Furthermore, we had to use $\mu_c$ instead of $\mu$ because of equality \eqref{infimum}; more precisely, it is not necessarily true that, for $A\in\sigma(\F_c)$ one has
\[
\mu(A)=\inf\left\{\sum_{i\ge 1}\mu(A_i): A_i\in\F_c\ \forall i,\ A\subset\bigcup_i A_i\right\}.
\] 
\end{remark}

\begin{lemma}\label{3.15}
Let $(\Omega,\B,\rho)$ be a measure space on which a group $G$ acts by measurable automorphisms, and such that $\rho(gB)=\rho(B)$ for every $B\in\B$ such that $\rho(B)<\infty$. Then the action of $G$ preserves $\rho$, \textit{i.e.} $\rho(gB)=\rho(B)$ for every $B\in\B$.
\end{lemma}
\begin{proof} If there existed $g\in G$ and $B\in\B$ such that $\rho(B)=\infty$ and $\rho(gB)\not=\rho(B)$, then necessarily, we would have $\rho(gB)<\infty$. But this contradicts our hypothesis since then $\rho(g^{-1}gB)=\rho(B)$ would be finite.
\end{proof}

\par\vspace{2mm}
The next proposition is the last step of the proof of Theorem \ref{thmC0}.

\begin{proposition}
 There exists a $G$-invariant subset $\Om\in\sigma(\F_c)$ such that the measure $\rho:\sigma(\F_c)\rightarrow [0,\infty]$, defined by $\rho(B):=\mu_c(B\cap\Om)$ for every $B\in\sigma(\F_c)$, is $\sigma$-finite and $G$-invariant. Furthermore, the Hilbert space $L^2(\Om,\rho)$ is separable and contains the sequence of unit vectors $(\xi_m)$ of Proposition \ref{resume}.
\end{proposition}
\begin{proof}
Let $D=\{g_1=e,g_2,\ldots\}\subset G$ be a countable, dense subset of $G$. Set
$Y=\bigcup_{m\geq 1}X_m$ where $X_m=\prod_{k\ge 1}A_{n(k,m)}$ as in Proposition \ref{resume}, and set
\[
\Om=\bigcup_{h\in D}hY.
\]
Recall that the  measure $\rho$ on $(X,\sigma(\F_c))$ is defined by
\[
\rho(B)=\mu_c(B\cap\Om)
\]
for every $B\in\sigma(\F_c)$, so that $\rho$ is $\sigma$-finite. It remains to prove that it is $G$-invariant. By Lemma \ref{3.15}, it suffices to prove that $\rho(gB)=\rho(B)$ for every $g\in G$ and $B\in\sigma(\F_c)$ such that $\rho(B)<\infty$. 
Then, for every $i\ge 1$, set
\[
B_i=(B\cap g_iY)\setminus\left(\bigcup_{j=1}^{i-1}g_jY\right),
\]
so that $B\cap\Om=\bigsqcup_i B_i$, and thus $\rho(B)=\sum_i\rho(B_i)<\infty$.

Then
\[
\rho(gB)\ge \rho(g(B\cap\Om))=\sum_i \rho(gB_i).
\]
For fixed $i$, let $(h_n^{(i)})_{n\ge 1}\subset D$ be such that $h_n^{(i)}\to gg_i$ as $n\to\infty$. Then $\rho(gB_i)\ge \rho(gB_i\cap h_n^{(i)}g_i^{-1}B_i)$ for every $n$, and 
\[
\rho(gB_i\cap h_n^{(i)}g_i^{-1}B_i)=\mu_c(gB_i\cap h_n^{(i)}g_i^{-1}B_i)\to \mu_c(gB_i)=\mu_c(B_i)
\]
as $n\to\infty$, thus $\rho(gB_i)\ge\mu_c(B_i)$, hence 
\[
\rho(gB)\ge \sum_i\rho(gB_i)\ge \sum_i\mu_c(B_i)=\mu_c(B)\ge \mu_c(B\cap\Om)=\rho(B).
\]
We also get
\[
\rho(B)\le \rho(gB)\le \rho(g^{-1}gB)=\rho(B).
\]
Finally, separability of $L^2(\Om,\rho)$ follows from the countability of $D$ and from its density in $G$.

The proof is now complete.
\end{proof}

\par\vspace{3mm}

\bibliographystyle{plain}
\bibliography{refFixed}

\end{document}